\documentclass[
11pt,  reqno]{amsart}

\usepackage{amsmath,amssymb,amscd,amsthm}

\usepackage{color}

\usepackage{hyperref}

\usepackage{cases}

\usepackage{bm}

\allowdisplaybreaks[3]

\newtheorem{theorem}{Theorem} [section]

\newtheorem{lemma}[theorem]{Lemma}
\newtheorem{proposition}[theorem]{Proposition}
\newtheorem{remark}[theorem]{Remark}
\newtheorem{definition}[theorem]{Definition}
\newtheorem{corollary}[theorem]{Corollary}

\DeclareMathOperator*{\supp}{supp}

\newcommand{\noi}{\noindent}
\newcommand{\Z}{\mathbb{Z}}
\newcommand{\R}{\mathbb{R}}
\newcommand{\C}{\mathbb{C}}
\newcommand{\T}{\mathbb{T}}

\renewcommand{\L}{\mathcal{L}}

\newcommand{\al}{\alpha}

\newcommand{\eps}{\varepsilon}

\newcommand{\G}{\Gamma}
\newcommand{\ld}{\lambda}
\newcommand{\Ld}{\Lambda}
\newcommand{\s}{\sigma}

\newcommand{\ft}{\widehat}
\newcommand{\Ft}{{\mathcal{F}}}
\newcommand{\wt}{\widetilde}

\newcommand{\dx}{\partial_x}
\newcommand{\dy}{\partial_y}
\newcommand{\dt}{\partial_t}

\renewcommand{\l}{\ell}

\newcommand{\les}{\lesssim}

\newcommand{\jb}[1]
{\langle #1 \rangle}

\newcommand{\N}{\mathbb{N}}

\newcommand{\Aa}{\mathcal M_0}

\newcommand{\Nn}{\mathcal N}
\newcommand{\Bn}{\mathcal B}

\newcommand{\uu}{\mathfrak u}

\newcommand{\tu}{w}

\DeclareMathOperator*{\smm}{\max{\hspace*{-2pt}}^{(2)}}

\numberwithin{equation}{section}
\numberwithin{theorem}{section}

\let\Re=\undefined\DeclareMathOperator*{\Re}{Re}
\let\Im=\undefined\DeclareMathOperator*{\Im}{Im}

\begin{document}
\baselineskip = 14pt

\title[Norm inflation for a periodic derivative higher-order NLS]
{Norm inflation for
a higher-order
\\
nonlinear Schr\"odinger equation
\\
with a derivative
on the circle
}

\author[T.~Kondo and M.~Okamoto]
{Toshiki Kondo and Mamoru Okamoto}

\address{
Toshiki Kondo\\
Department of Mathematics\\
Graduate School of Science\\ Osaka University\\
Toyonaka\\ Osaka\\ 560-0043\\ Japan
}
\email{u534463k@ecs.osaka-u.ac.jp}

\address{
Mamoru Okamoto\\
Department of Mathematics\\
Graduate School of Science\\ Osaka University\\
Toyonaka\\ Osaka\\ 560-0043\\ Japan}
\email{okamoto@math.sci.osaka-u.ac.jp}



\begin{abstract}
We consider a periodic higher-order nonlinear Schr\"odinger equation
with the nonlinearity $u^k \dx u$,
where $k$ is a natural number.
We prove the norm inflation in a subspace of the Sobolev space $H^s(\T)$ for any $s \in \R$.
In particular,
the Cauchy problem is ill-posed in $H^s(\T)$ for any $s \in \R$.
\end{abstract}

\date{\today}
\maketitle
%


\section{Introduction}

\label{SEC:1}

We consider the Cauchy problem for the following higher-order nonlinear Schr\"odinger equation (NLS)
with a derivative:
\begin{equation} 
\left\{
\begin{aligned}
&\dt u - i (-\dx^2)^{\frac \al2} u = \ld u^k \dx u, \\ 
&u|_{t=0} = \phi,
\end{aligned}
\right.
\label{NLSka}
\end{equation}
where $\al>0$, $\ld \in \C \setminus \{ 0 \}$, $k \in \N$,
and $(-\dx^2)^{\frac \al2}$ denotes the Fourier multiplier with the symbol $|n|^\al$
(See the end of this section for notation).
Here,
$\T := \R / 2\pi \Z$, $u = u(t,x) : \R \times \T \to \C$ is an unknown function,
and $\phi : \T \to \C$ is a given function.
Our main goal in this paper is to prove the ill-posedness of \eqref{NLSka}.

The linear case
\eqref{NLSka} with $k=0$, namely,
\begin{equation}
\left\{
\begin{aligned}
&\dt u - i (-\dx^2)^{\frac \al2} u = \ld \dx u, \\ 
&u|_{t=0} = \phi,
\end{aligned}
\right.
\label{NLSk0}
\end{equation}
is well-posed in $L^2(\T)$
if and only if $\Im \ld =0$.
Since \eqref{NLSk0} has constant coefficients,
this equivalence follows from a simple observation.
See \cite{Chi02, Miz06, Miz85} for variable coefficients.
Indeed,
by setting
\begin{equation}
v(t,x) := e^{- it (-\dx^2)^{\frac \al2}} u(t,x),
\label{norv1}
\end{equation}
\eqref{NLSk0} is equivalent to
\begin{equation} 
\left\{
\begin{aligned}
&\dt v = \ld \dx v, \\ 
&v|_{t=0} = \phi.
\end{aligned}
\right.
\label{NLSk00}
\end{equation}
Accordingly,
$v(t,x) = \phi (x+ \ld t)$ solves \eqref{NLSk00} for $\Im \ld =0$,
which implies the well-posedness in $L^2(\T)$.
Moreover,
for $N \in \N$ and $\phi (x) =  e^{iNx}$,
the solution to \eqref{NLSk00} is $v(t,x) = e^{i t \ld N} e^{i N x}$.
Then, since $\| v(t) \|_{L^2} = e^{- t (\Im \ld) N}$,
\eqref{NLSk00} is ill-posed in $L^2(\T)$ if $\Im \ld \neq 0$.

For \eqref{NLSka} with $k \in \N$,
by taking the transformation $u \mapsto \ld^{\frac 1k} u$,
we may assume $\ld = 1$.
In what follows,
we only consider \eqref{NLSka} with $k \in \N$ and $\ld=1$:
\begin{equation} 
\left\{
\begin{aligned}
&\dt u - i (-\dx^2)^{\frac \al2} u = u^k \dx u, \\ 
&u|_{t=0} = \phi.
\end{aligned}
\right.
\label{NLSk}
\end{equation}

When $\al=2$,
Chihara \cite{Chi02}
proves the ill-posedness in the Sobolev space $H^s(\T)$ for any $s \in \R$.
Moreover,
Christ \cite{Chr03} shows the norm inflation
with infinite loss of regularity.
Namely,
for any $s, \s \in \R$,
a solution with a smooth initial data $\phi$ and $\| \phi \|_{H^s} \ll 1$
exhibits a large $H^\s$-norm in a short time.
On the other hand,
Chung, Guo, Kwon, and Oh \cite{CGKO17} prove
the well-posedness in $L^2(\T)$ under the mean-zero and smallness assumptions
when $\al=2$ and $k=1$.

We emphasize that the structure of nonlinearity plays an important role in obtaining well-posedness.
Indeed,
by using the energy method,
Ambrose and Simpson \cite{AmSi15}
prove the well-posedness in $H^2(\T)$ of the Cauchy problem
for the generalized derivative
NLS
\[
\left\{
\begin{aligned}
&\dt u + i \dx^2 u = |u|^k \dx u,
\\
&u|_{t=0} = \phi
\end{aligned}
\right.
\]
for $k \ge 2$.
See also \cite{DNY21, Hao07, Her06, KiTs18,  KiTs23, Tak99}
for periodic NLS with a derivative.
Note that the energy method does not work for \eqref{NLSk}.
See also Remark \ref{rem:KdV} below.

Before stating the main result,
we define a solution to \eqref{NLSk}.

\begin{definition}
\label{def:sol}
Let $s \in \R$, $T > 0$, and $\phi \in H^s(\T)$.
We say that $u$ is a solution to \eqref{NLSk} in $H^s(\T)$ on $[0,T]$
if $u$ satisfies the followings:
\begin{enumerate}
\item
$u \in C([0,T]; H^s(\T)) \cap L^{k+1}_{\text{loc}} ([0,T) \times \T)$,

\item
For any $\chi \in C_c^\infty([0,T) \times \T)$,%
\footnote{Here, $C_c^\infty([0,T) \times \T)$ denotes the space of smooth functions with compact support in $[0,T) \times \T$.
}
we have
\begin{align*}
&- \int_0^T \int_{-\pi}^\pi u (t,x) \dt \chi(t,x) dx dt
- \int_{-\pi}^\pi \phi (x) \chi (0,x) dx
\\
&\quad
-i \int_0^T \int_{-\pi}^\pi u(t,x) (-\dx^2)^{\frac \al 2} \chi(t,x) dx dt
\\
&=
-\frac 1{k+1}
\int_0^T \int_{-\pi}^\pi u(t,x)^{k+1} \dx \chi (t,x) dx dt.
\end{align*}
\end{enumerate}
\end{definition}

The condition (ii) in Definition \ref{def:sol} means that
$u$ satisfies \eqref{NLSk} in the sense of distribution.

\begin{remark}
\rm
\label{rem:sol1}
Let $u \in C([0,T]; H^s(\T))$ be a solution to \eqref{NLSk}.
When $s>\frac 12$,
it holds that
\[
u^{k+1} \in C([0,T]; H^s(\T)),
\quad
u^k \dx u = \frac 1{k+1} \dx(u^{k+1}) \in C([0,T]; H^{s-1}(\T)).
\]
Accordingly,
if $\al \ge 1$ and $s>\frac 12$,
we have
$u \in C^1([0,T]; H^{s-\al}(\T))$
and
\[
\dt u -i (-\dx^2)^{\frac \al2} u = u^k \dx u
\]
holds in $H^{s-\al}(\T)$ for every $t \in [0,T]$.
\end{remark}

For $s \in \R$, define
\[
H^s_{\ge 0}(\T)
:=
\{ f \in H^s (\T) \mid \ft f(n) =0 \text{ for } n<0 \}.
\]
Note that $H^s_{\ge 0} (\T)$ is a closed subspace of $H^s(\T)$.
The following is the main result in the present paper.  

\begin{theorem}
\label{th:1.1}
Assume that $\al=2$ or $\al \ge 3$.
Set
\begin{equation}
s_0 := \begin{cases} 2 & \text{ if } \al=2, \\
1 & \text{ if } \al \ge 3.
\end{cases}
\label{s0}
\end{equation}
Let $k \in \N$, $s \ge s_0$, and $\s \in \R$.
Then,
for any $\eps> 0$,
there exist an initial data $\phi \in C^\infty(\T)$ with $\|\phi\|_{H^{s}(\T)} < \eps$
and
a time $T \in (0,\eps)$
satisfying one of the following:
\begin{enumerate}
\item
there is no solution $u \in C([0,T]; H^{s}_{\ge 0}(\T))$ to \eqref{NLSk},

\item
there is a solution $u \in C([0,T]; H^{s}_{\ge 0}(\T))$ to \eqref{NLSk} such that
\[
\|u(T)\|_{H^\s} > \eps^{-1}.
\]
\end{enumerate}
\end{theorem}

Theorem \ref{th:1.1} shows
the norm inflation
with infinite loss of regularity.
In particular,
the flow map in $H^s(\T)$ for $s \in \R$, if exists, is not a continuous extension of that in $H^{\max (s,s_0)}_{\ge 0} (\T)$.
In this sense, \eqref{NLSk} is ill-posed in $H^s(\T)$ for any $s \in \R$.
This is a generalization of the result by \cite{Chi02, Chr03}.%
\footnote{Strictly speaking,
a solution in \cite{Chr03} may differ from that in Definition \ref{def:sol}.
In fact, the solution in \cite{Chr03} might have Fourier coefficients that increase exponentially.
See (3.8) in \cite{Chr03}.
}
In other words,
Theorem \ref{th:1.1} says that
the derivative loss of \eqref{NLSk} on $\T$
never recover
even for the higher-order NLS,
which is a sharp contrast on $\R$.
In fact,
\eqref{NLSk}
is well-posed in $H^s(\R)$ for some $s$.
See \cite{HO95, HIT20, KPV93, Oza98, Por08}, for example.

Since it is unclear whether a solution to \eqref{NLSk} exists even if initial data are smooth,
the case (i) in Theorem \ref{th:1.1} might happen.
Hence,
Theorem \ref{th:1.1} implies
that
the flow map of \eqref{NLSk}, if exists,
is discontinuous in $H^s_{\ge 0}(\T)$ for $s \ge s_0$.
On the other hand,
we show the case (ii) in Theorem \ref{th:1.1} for $\s<s_0$
assuming the existence of a solution in $H^{s_0}_{\ge 0}(\T)$ (not in $H^\s_{\ge 0} (\T)$).
Namely,
Theorem \ref{th:1.1} asserts non-existence of a continuous extension of the flow map in $H^{s_0}_{\ge 0}(\T)$ to $H^\s_{\ge 0}(\T)$ for $\s<s_0$.

In order to prove Theorem \ref{th:1.1},
we show that a similar situation arises with the linear equation \eqref{NLSk0}.
For $N \in \N$,
we consider a solution $u$ to \eqref{NLSk}
with $\phi \in H^s_{\ge 0}(\T)$ satisfying
\[
\ft \phi(n)=0
\]
for $n=1,\dots, N-1$.
From \eqref{NLSk},
we have
$\dt \ft u(t,0) =0$, 
namely,
$\ft u(t,0) = \ft \phi(0)$
for $0 \le t \le T$.
Once
we obtain
\begin{equation}
\ft u(t,n) =0 
\quad \text{ for $n=1,\dots, N-1$},
\label{unexp}
\end{equation}
then
$\ft u(t,N)$ satisfies
\[
\dt \ft u(t,N) -i N^\al \ft u(t,N) = iN \ft \phi(0)^k \ft u(t,N).
\]
Namely,
$\ft u(t,N) = e^{it \ft \phi(0)^k N} e^{it N^\al}$.
Hence,
if $\Im ( \ft \phi(0)^k) <0$,
we obtain the desired result.
However,
it is not clear whether \eqref{unexp} holds for a solution to \eqref{NLSk} in the sense of Definition \ref{def:sol}.

To obtain \eqref{unexp},
we employ the unconditional uniqueness of \eqref{NLSk} in $H^{s_0}_+(\T)$,
where $s_0$ is given in \eqref{s0}.
Here,
$H^s_+(\T)$ denotes
\[
H^s_+ (\T)
:=
\{ f \in H^s_{\ge 0} (\T) \mid \ft f(0) =0 \}
\]
for $s \in \R$.
Moreover,
``unconditional'' means that uniqueness of the solution in the sense of Definition \ref{def:sol} holds in the entire space
$C([0,T]; H^{s_0}_+(\T))$.
Since the dispersive effect in the nonlinear terms does not vanish if $\phi \in H^s_+(\T)$,
we can recover a derivative loss in \eqref{NLSk}.
See \cite{CLW23, NaWa}
for well-posedness results in a space of distributions whose Fourier support is in the half space.

We prove the unconditional uniqueness of \eqref{NLSk} in $H^{s_0}_+ (\T)$ in Section \ref{sec:AWP}.
The unconditional uniqueness for $\al \ge 3$ follows from the normal form reduction as in \cite{Arn88, Nik86}.
Although the abstract framework in \cite{Kis19} applies to \eqref{NLSk},
we give a proof in Subsection \ref{subsec:NFR} for readers' convenience.
By applying an infinite iteration scheme of normal form reductions as in \cite{CGKO17},
we might obtain the unconditional uniqueness for $\al=2$.
However,
to avoid some technical difficulties,
we use a gauge transformation for $\al=2$ as in \cite{Oza98}
instead of the normal form reduction.
See also \cite{Chi94, HO95}.

\begin{remark}
\rm
We expect that the unconditional uniqueness holds for $2<\al<3$,
if we apply the normal form reduction many times.
However,
since we focus on the ill-posedness of the higher-order NLS,
we do not pursue the case $2<\al<3$ in this paper.
\end{remark}

\begin{remark}
\label{rem:KdV}
\rm

We can replace
$(-\dx^2)^{\frac \al2}$ in \eqref{NLSk} by
$-i(-\dx^2)^{\frac{\al-1}2} \dx$,
since we consider a solution in $H^s_{\ge 0}(\T)$.
Namely,
the same norm inflation result as in Theorem \ref{th:1.1} holds for
the higher-order Benjamin-Ono or Korteweg-de Vries equation:
\begin{equation} 
\left\{
\begin{aligned}
&\dt u - (-\dx^2)^{\frac{\al-1}2} \dx u = u^k \dx u, \\ 
&u|_{t=0} = \phi.
\end{aligned}
\right.
\label{KdVk}
\end{equation}
Note that we consider a solution $u \in C([0,T]; H^s_{\ge 0}(\T))$ to \eqref{KdVk}.
In particular, $u$ is a complex-valued function.
This assumption drastically changes the structure on the periodic setting.
In fact,
there are many well-posedness results for \eqref{KdVk} when $u$ is real-valued.
See
\cite{Cha21, GKT23, Jo24, KaTo06, KePi16, KLV24, KiSc21, MoTa22, MoPi23},
for example.
See also
\cite{Arn19, Hur18, Mol12}
for some negative results on the real-valued setting.

\end{remark}

\mbox{}

\noi
\textbf{Notation.}
Given an integrable function $f$,
define
\begin{align*}
\int_\T f(x) dx
&:= \frac 1{2\pi} \int_{-\pi}^\pi f(x) dx,
\\
\Ft[f](n)
=
\ft f(n)
&:=\int_{\T} f(x) e^{-inx} dx
\end{align*}
for $n \in \Z$.
Moreover,
let $\Ft [f]$ or $\ft f$ denote the Fourier coefficient of a periodic distribution $f$.
Note that the Fourier series expansion
\[
f(x) = \sum_{n \in \Z} \ft f(n) e^{inx}
\]
(in the sense of distribution)
holds for a periodic distribution $f$.
For $\al>0$ and $t \in \R$, we define
\begin{align*}
(-\dx^2)^{\frac \al2} f (x)
&:= \sum_{n \in \Z} |n|^\al \ft f(n) e^{inx},
\\
e^{it(-\dx^2)^{\frac \al2}} f (x)
&:= \sum_{n \in \Z} e^{it |n|^\al} \ft f(n) e^{inx}.
\end{align*}

We also use the notation $\ft u(t,n)$ to express the Fourier coefficient with respect to $x$ of a two variable function $u(t,x)$.
Note that $(-\dx^2)^{\frac \al2}$ acts only on $x$,
namely,
\[
(-\dx^2)^{\frac \al2} u(t,x)
:= \sum_{n \in \Z} |n|^\al \ft u(t,n) e^{inx}.
\]

For $s \in \R$,
we define $H^s(\T)$ to be the set of all periodic distributions
for which the norm
\[
\| f \|_{H^s}
:=
\bigg(\sum_{n \in \Z} \jb{n}^{2s} |\ft f(n)|^2 \bigg)^{\frac12}
\]
is finite,
where $\jb{n} := \sqrt{1 + n^2}$ for $n \in \Z$.
Note that Parseval's identity implies $\| f \|_{H^0} = \| f \|_{L^2}$.
In particular, we have $H^0(\T) = L^2(\T)$.

We use the notation $A \les B$ if there is a constant $C>0$
(depending only on $\al$, $k$, $s$, and $\s$ in Theorem \ref{th:1.1})
such that $A \le C B$.
We also denote $A \ll B$ when $A \le CB$ with sufficiently small $C>0$.

\section{Unconditional uniqueness}
\label{sec:AWP}

In this section,
we prove the
unconditional uniqueness
of \eqref{NLSk}.

\begin{proposition}
\label{prop:welH+}
Assume that $\al=2$ or $\al \ge 3$.
Let $k \in \N$ and $\phi \in H^{s_0}_+(\T)$ with $\| \phi \|_{H^{s_0}} \ll 1$,
where $s_0$ is defined by \eqref{s0}.
Then,
for any $T \in (0,1]$,
a solution $u \in C([0,T]; H^{s_0}_+(\T))$ to \eqref{NLSk} is unique.
Moreover,
$u$ satisfies
\begin{equation}
\ft u(t,n) =0
\label{supp0}
\end{equation}
unless there exist $m \in \N$ and $n_1, \dots, n_m \in \supp \ft \phi$
such that $n= \sum_{\l=1}^m n_\l$.
\end{proposition}

The regularity assumption in Proposition \ref{prop:welH+} is not optimal;
nevertheless, it suffices to prove our main result,
since \eqref{supp0} holds.
Note that there is no requirement for $T$ to be small.

We also consider the Cauchy problem \eqref{NLSk} with $u^k$ replaced by a polynomial of $u$:
\begin{equation} 
\left\{
\begin{aligned}
&\dt u - i (-\dx^2)^{\frac \al2} u = \Big( \sum_{\l=1}^k \ld_\l u^\l \Big) \dx u, \\ 
&u|_{t=0} = \phi,
\end{aligned}
\right.
\label{NLSkg}
\end{equation}
where $\ld_1, \dots, \ld_k$ are complex constants.

\begin{corollary}
\label{cor:UU}
%
The same statement as in Proposition \ref{prop:welH+} is true for \eqref{NLSkg}.
\end{corollary}

Since the proof of Corollary \ref{cor:UU} is a straightforward adaptation,
we only prove Proposition \ref{prop:welH+} in the next subsections.

\subsection{Normal form reduction}
\label{subsec:NFR}

In this subsection, we consider the case $\al \ge 3$ in Proposition \ref{prop:welH+}.
We use the normal form reduction as in \cite{Arn88, Kis19, Nik86}.

Let $u \in C([0,T]; H^{s_0}_+(\T))$ 
be a solution to \eqref{NLSk}.
Set
$v$ as in \eqref{norv1}.
Since $u$ solves \eqref{NLSk}
(see also Remark \ref{rem:sol1}),
$v$ satisfies
\begin{equation}
\dt \ft v(t,n)
= \frac{in}{k+1}
\sum_{\substack{n_1, \dots, n_{k+1} \in \N
\\
n_1+ \dots + n_{k+1} =n
}}
e^{it \Phi (n_1, \dots, n_{k+1})}
\prod_{\l=1}^{k+1}
\ft v(t,n_\l),
\label{vn1}
\end{equation}
where $\Phi(n_1, \dots, n_{k+1})$ is defined by
\[
\Phi(n_1, \dots, n_{k+1})
:=
- \bigg( \sum_{\l=1}^{k+1} n_\l \bigg)^{\al}
+ \sum_{\l=1}^{k+1} n_\l^{\al}.
\]

\begin{lemma}
\label{lem:Phi1}
Let $\al \ge 1$ and $k \in \N$.
Then,
$\Phi$ defined above satisfies
\begin{equation}
|\Phi(n_1, \dots, n_{k+1})|
\ge
(\al-1)
\Big( \max_{\l= 1, \dots, k+1} n_\l \Big)^{\al-1}
\Big( \smm_{\l= 1, \dots, k+1} n_\l \Big)
\label{phase2}
\end{equation}
for $n_1, \dots, n_{k+1} \in \N$,
where $\max$ and $\smm$ denote the largest and second largest elements, respectively.
\end{lemma}

\begin{proof}
Without loss of generality,
we may assume that $n_1 \ge \dots \ge n_{k+1} \ge 1$.
We employ an induction argument.
First, we consider the case $k=1$.
A direct calculation with $n_1 \ge n_2 \ge 1$ yields that
\begin{align*}
|\Phi(n_1,n_2)|
&= (n_1+n_2)^\al - n_1^\al - n_2^\al
= n_2 \int_0^1 \al (n_1+\theta n_2)^{\al-1} d\theta - n_2^\al
\\
&\ge
(\al n_1^{\al-1} - n_2^{\al-1}) n_2
\ge
(\al-1)
n_1^{\al-1} n_2.
\end{align*}
Assume that \eqref{phase2} holds up to $k-1$ for $k \ge 2$.
Then, we have
\begin{align*}
|\Phi(n_1, \dots, n_{k+1})|
&=
\bigg( \sum_{\l=1}^{k+1} n_\l \bigg)^{\al}
- n_1^\al - \bigg( \sum_{\l=2}^{k+1} n_\l \bigg)^\al
\\
&\quad
+ \bigg( \sum_{\l=2}^{k+1} n_\l \bigg)^\al
- \sum_{\l=2}^{k+1} n_\l^\al
\\
&\ge
(\al-1)
\max \bigg( n_1, \sum_{\l=2}^{k+1} n_\l\bigg)^{\al-1}
\min \bigg( n_1, \sum_{\l=2}^{k+1} n_\l \bigg)
\\
&\quad
+
(\al-1)
n_2^{\al-1}
n_3
\\
&\ge
(\al-1)
n_1^{\al-1}
n_2
,
\end{align*}
which concludes the proof.
\end{proof}

In particular, $\Phi(n_1, \dots, n_{k+1}) \neq 0$ holds.
By \eqref{vn1}, we have
\begin{equation}
\begin{aligned}
\dt \ft v(t,n)
&=
\dt
\bigg(
\frac{n}{k+1}
\sum_{\substack{n_1, \dots, n_{k+1} \in \N
\\
n_1+ \dots + n_{k+1} =n
}}
\frac{e^{it \Phi (n_1, \dots, n_{k+1})}}{\Phi (n_1, \dots, n_{k+1})}
\prod_{\l=1}^{k+1}
\ft v(t,n_\l)
\bigg)
\\
&\quad
-
\frac{n}{k+1}
\sum_{\substack{n_1, \dots, n_{k+1} \in \N
\\
n_1+ \dots + n_{k+1} =n
}}
\frac{e^{it \Phi (n_1, \dots, n_{k+1})}}{\Phi (n_1, \dots, n_{k+1})}
\dt
\bigg( \prod_{\l=1}^{k+1}
\ft v(t,n_\l)
\bigg)
.
\end{aligned}
\label{nov1}
\end{equation}
Define
\begin{equation}
\Nn_n (v)(t)
:=
\frac{n}{k+1}
\sum_{\substack{n_1, \dots, n_{k+1} \in \N
\\
n_1+ \dots + n_{k+1} =n
}}
\frac{e^{it \Phi (n_1, \dots, n_{k+1})}}{\Phi (n_1, \dots, n_{k+1})}
\prod_{\l=1}^{k+1}
\ft v(t,n_\l).
\label{Nn1}
\end{equation}
In the second part on the right-hand side of \eqref{nov1},
it may assume that the time derivative in $\dt
( \prod_{\l=1}^{k+1}
\ft v(t,n_\l)
)$ falls only on $\ft v(t,n_{k+1})$.
Then,
it follows from \eqref{nov1}, \eqref{Nn1}, and \eqref{vn1}, that
\begin{equation}
\begin{aligned}
&
\dt \ft v(t,n)
- \dt \Nn_n(v)(t)
\\
&=
-
\frac{n}{k+1}
\sum_{\substack{n_1, \dots, n_{k+1} \in \N
\\
n_1+ \dots + n_{k+1} =n
}}
\frac{e^{it \Phi (n_1, \dots, n_{k+1})}}{\Phi (n_1, \dots, n_{k+1})}
\bigg(
\prod_{\l=1}^{k}
\ft v(t,n_\l)
\bigg)
\\
&\hspace*{20pt}
\times
in_{k+1}
\sum_{\substack{m_1, \dots, m_{k+1} \in \N
\\
m_1+ \dots + m_{k+1} =n_{k+1}
}}
e^{it \Phi (m_1, \dots, m_{k+1})}
\bigg(
\prod_{\l'=1}^{k+1}
\ft v(t, m_{\l'})
\bigg)
\\
&=: \Bn_n(v)(t)
.
\end{aligned}
\label{Bn1}
\end{equation}
We also define
\begin{align*}
\Nn (v)(t,x)
&:=
\sum_{n=1}^\infty \Nn_n (v)(t) e^{inx},
\\
\Bn (v)(t,x)
&:=
\sum_{n=1}^\infty \Bn_n (v)(t) e^{inx}.
\end{align*}
Then,
$v$ satisfies the integral equation
\begin{equation}
v(t)
= \phi + \Nn(v)(t) - \Nn(\phi) (0)
+ \int_0^t \Bn(v)(t') dt'.
\label{norv2}
\end{equation}
We solve this integral equation by using the contraction argument.

\begin{proposition}
\label{prop:nor2}
Let $\al \ge 3$, $k \in \N$, and $\phi \in H^1_+(\T)$ with $\| \phi \|_{H^1} \ll 1$.
Then,
for any $T \in (0,1]$,
there exists a unique solution $v \in C([0,T]; H^1_+(\T))$ to \eqref{norv2}.
\end{proposition}

\begin{proof}
First, we show the existence of a solution.
From
\eqref{Nn1}, \eqref{phase2}, and $\al \ge 3$,
and Young's convolution inequality,
we have
\[
\begin{aligned}
\| \Nn (v)(t) \|_{H^1}
&\les
\big\| |\ft v(t)| \ast \dots \ast |\ft v(t)| \big\|_{\l^2_n}
\\
&\les
\| \ft v(t) \|_{\l^2_n}
\| \ft v(t) \|_{\l^1_n}^{k}
\les
\| v(t) \|_{H^1}^{k+1}
\end{aligned}
\]
for $v \in C([0,T]; H^1_+(\T))$ and $0 \le t \le T$.
The same calculation with \eqref{Bn1} yields that
\begin{align*}
\| \Bn (v)(t) \|_{H^1}
&\les
\big\|
\jb{n} (|\ft v(t)| \ast \dots \ast |\ft v(t)|) \big\|_{\l^2_n}
\\
&\les
\| \jb{n} \ft v(t) \|_{\l^2_n}
\| \ft v(t) \|_{\l^1_n}^{2k}
\les
\| v(t) \|_{H^1}^{2k+1}
\end{align*}
for $v \in C([0,T]; H^1_+(\T))$ and $0 \le t \le T$.
Moreover,
we obtain that
\begin{align}
&
\begin{aligned}
&\| \Nn (v_1)(t) - \Nn (v_2)(t) \|_{H^1}
\\
&\les
\| v_1(t)-v_2(t) \|_{H^1}
\big(
\| v_1(t) \|_{H^1} + \| v_2(t) \|_{H^1}
\big)
^{k},
\end{aligned}
\label{eNn3}
\\
&
\begin{aligned}
&\| \Bn (v_1)(t) - \Bn (v_2)(t) \|_{H^1}
\\
&\les
\| v_1(t)-v_2(t) \|_{H^1}
\big(
\| v_1(t) \|_{H^1} + \| v_2(t) \|_{H^1}
\big)
^{2k}
\end{aligned}
\label{eNn4}
\end{align}
for $v_1, v_2 \in C([0,1]; H^1_+(\T))$ and $0 \le t \le T$.

Define
\begin{equation}
\G (v) (t)
:=
\phi
+ \Nn (v)(t) - \Nn (\phi)(0)
+ \int_0^t \Bn (v)(t') dt'
\label{ConG1}
\end{equation}
for $v \in C([0,T]; H^1_+(\T))$.
It follows from \eqref{eNn3}, \eqref{eNn4}, and $0<T \le 1$
that $\G$ is a contraction mapping on
\[
\Big\{ v \in C([0,T]; H^1_+(\T)) \mid \sup_{0 \le t \le T} \| v (t) \|_{H^1} \le 2 \| \phi \|_{H^1} \Big\}
\]
provided that $\phi \in H^1_+(\T)$ satisfies $\| \phi \|_{H^1} \ll 1$.
The Banach fixed-point theorem shows that
there exists $v \in C([0,T]; H^1_+(\T))$
satisfying
$v = \G(v)$.
Namely, $v$ solves \eqref{norv2}.

Next, we show the (unconditional) uniqueness.
Let $v \in C([0,T]; H^1_+(\T))$ be a solution to \eqref{norv2}.
Define
\[
t_\ast := \sup \{ t \in [0,T] \mid \| v(t) \|_{H^1} < 2 \| \phi \|_{H^1} \}.
\]
By $v|_{t=0}=\phi$, we have $t_\ast >0$.
We prove $t_\ast =T$ by using a contradiction argument.
If $t_\ast <T$,
it follows from $v \in C([0,T]; H^1_+(\T))$ that $\| v(t_\ast) \|_{H^1} =2 \|\phi \|_{H^1}$.
Then,
\eqref{norv2} with \eqref{eNn3}, \eqref{eNn4}, and $\| \phi \|_{H^1} \ll 1$ yields that
\begin{align*}
\| v(t_\ast) \|_{H^1}
&\le \| \phi \|_{H^1}
+ C
\big( \| v(t_\ast) \|_{H^1}^{k+1} + \| \phi \|_{H^1}^{k+1}
+ \| v(t_\ast) \|_{H^1}^{2k+1} \big)
\\
&= \| \phi \|_{H^1}
+ C
\big( (2\| \phi  \|_{H^1})^{k+1} + \| \phi \|_{H^1}^{k+1}
+ (2\| \phi \|_{H^1})^{2k+1} \big)
\\
&< 2 \| \phi \|_{H^1}.
\end{align*}
This contradicts to $\| v(t_\ast) \|_{H^1} =2 \|\phi \|_{H^1}$.
Hence, we obtain $t_\ast =T$.
\end{proof}

\begin{proof}[Proof of Proposition \ref{prop:welH+} for $\al \ge 3$]
Let $u \in C([0,T]; H^1_+(\T))$ be a solution to \eqref{NLSk}.
Then,
$v$ defined in \eqref{norv1} solves \eqref{norv2}.
The uniqueness of a solution to \eqref{NLSk} follows from Proposition \ref{prop:nor2}.

From the proof of Proposition \ref{prop:nor2},
$v$ is a limit of the sequence $\{ v^{(\l)} \}_{\l \in \N}$ in $C([0,T]; H^1_+(\T))$
defined by
$v^{(1)} (t) := \phi$ and
\[
v^{(\l+1)} := \G(v^{(\l)})
\]
for $\l \in \N$,
where $\G$ is given in \eqref{ConG1}.
Since
$u (t) = e^{it (-\dx^2)^{\frac \al2}} v(t)$,
we obtain \eqref{supp0}.
\end{proof}

\subsection{Gauge transformation}
\label{subsec:GT}

In this subsection,
we prove Proposition \ref{prop:welH+} for $\al=2$
by using the gauge transformation as in \cite{Oza98}.

\begin{remark}
\label{rem:int01}
\rm
If $f \in H^1_+(\T)$,
then $f^k \in H^1_+(\T)$.
In particular,
$f^k$ is well-defined and
\[
\int_{-\pi}^\pi f(x)^k dx =0.
\]
Namely,
$\int_0^x f(y)^k dy \in H^2_+(\T)$.
\end{remark}

Set
\[
\L := \dt + i \dx^2.
\]
For suitable functions $f$ and $\Ld$,
we have
\begin{equation}
e^{\Ld} \L (e^{-\Ld} f)
=
\L f
+ (-\L \Ld + i (\dx \Ld)^2) f
-2i (\dx \Ld) \dx f.
\label{cal1}
\end{equation}

Let $u \in C([0,T]; H^2_+(\T))$ be a solution to \eqref{NLSk} with $\al=2$.
We have
\begin{equation}
\L \dx u
= k u^{k-1} (\dx u)^2 + u^k \dx^2u.
\label{du1}
\end{equation}
Set
\begin{equation}
\Ld (t,x) := \frac 1{2i} \int_0^x u(t,y)^k dy.
\label{Ld1}
\end{equation}
By Remark \ref{rem:int01},
this primitive is well-defined.
A direct calculation shows that
\begin{align*}
\dt \Ld (t,x)
&= \frac k{2i} \int_0^x u(t,y)^{k-1} \dt u(t,y) dy
\\
&= \frac k{2i}\int_0^x \big( -i u(t,y)^{k-1} \dy^2 u(t,y)
+ u(t,y)^{2k-1} \dy u(t,y) \big)
dy
\\
&= -\frac k{2}
\bigg(
u(t,x)^{k-1} \dx u(t,x) - u(t,0)^{k-1} \dx u(t,0)
\\
&\hspace*{40pt}
-
(k-1)
\int_0^x u(t,y)^{k-2} (\dy u(t,y))^2 dy
\bigg)
\\
&\quad
+
\frac 1{4i}
\big(
u(t,x)^{2k} - u(t,0)^{2k} \big).
\end{align*}
Note that the third term on the right-hand side disappears when $k=1$.
Hence,
we have
\begin{equation}
\begin{aligned}
&\L \Ld (t,x)
\\
&=
\frac k{2}
\bigg(
u(t,0)^{k-1} \dx u(t,0)
+
(k-1)
\int_0^x u(t,y)^{k-2} (\dy u(t,y))^2 dy
\bigg)
\\
&\quad
+
\frac 1{4i}
\big(
u(t,x)^{2k} - u(t,0)^{2k} \big).
\end{aligned}
\label{LLd}
\end{equation}
From \eqref{du1}, \eqref{Ld1}, \eqref{LLd} and \eqref{cal1} with $f=\dx u$,
we obtain
\begin{equation}
\begin{aligned}
&e^{\Ld} \L (e^{-\Ld} \dx u) (t,x)
\\
&=
k \big( u^{k-1} (\dx u)^2 \big)(t,x)
\\
&\quad
-
\frac k{2}
\bigg(
\big( u^{k-1} \dx u \big)(t,0)
\\
&\hspace*{50pt}
+
(k-1)
\int_0^x (u^{k-2} \big( \dy u)^2 \big)(t,y) dy
\bigg)
\dx u (t,x)
\\
&\quad
+
\frac 1{4i}
u(t,0)^{2k}
\dx u (t,x).
\end{aligned}
\label{uu1}
\end{equation}

Set
\begin{equation}
\uu := e^{-\Ld} \dx u.
\label{uug}
\end{equation}
From \eqref{NLSk} with $\al =2$ and \eqref{uu1},
$u$ and $\uu$ satisfy
\begin{equation}
\begin{aligned}
\L u (t,x) &= \big( u^k e^{\Ld} \uu \big)(t,x),
\\
\L \uu (t,x) &=
k \big( u^{k-1} e^{\Ld} \uu^2\big)(t,x)
\\
&\quad
-
\frac k{2}
\bigg(
\big( u^{k-1} e^{\Ld} \uu \big) (t,0)
\\
&\hspace*{50pt}
+
(k-1)
\int_0^x \big(u^{k-2} e^{2 \Ld } \uu^2 \big) (t,y) dy
\bigg)
\uu (t,x)
\\
&\quad
+
\frac 1{4i}
u(t,0)^{2k}
\uu (t,x).
\end{aligned}
\label{gasys}
\end{equation}
Here, $\Ld$ is defined in \eqref{Ld1}.
Note that
the nonlinear parts of the system \eqref{gasys} have no derivatives.
Hence,
(even without the condition \eqref{uug},)
the standard contraction mapping theorem
yields the following.

\begin{proposition}
\label{prop:gagesys2}
Let $k \in \N$ and $\phi, \psi \in H^1_+(\T)$ with
$\| \phi \|_{H^1} + \| \psi \|_{H^1} \ll 1$.
Then,
for any $T \in (0,1]$,
there exists a unique solution
\[
(u,\uu) \in C([0,T]; H^1_+(\T) \times H^1_+(\T))
\]
to
\eqref{gasys}
with $(u,\uu)_{|t=0} (\phi, \psi)$.
\end{proposition}

\begin{proof}
Set
\[
X_T := C([0,T]; H^1_+(\T))
\]
equipped with the norm
\[
\| f \|_{X_T} := \sup_{0 \le t \le T} \| f(t) \|_{H^1}
\]
for $f \in X_T$.
Define
\[
\G(u,\uu) := (\G_1(u,\uu), \G_2(u,\uu))
\]
for $u, \uu \in X_T$,
where
\begin{align*}
\G_1(u,\uu) (t,x)
&:=
e^{-it \dx^2} \phi(x)
+ \int_0^t e^{-i (t-t') \dx^2} (u^k e^{\Ld} \uu)(t',x) dt',
\\
\G_2(u,\uu) (t,x)
&:=
e^{-it \dx^2} \psi (x)
+ \int_0^t e^{-i (t-t') \dx^2}
\bigg\{
k (u^{k-1} e^{\Ld} \uu^2)(t',x)
\\
&\quad
-
\frac k{2}
\bigg(
\big( u^{k-1} e^{\Ld} \uu \big) (t',0)
\\
&\hspace*{50pt}
+
(k-1)
\int_0^x \big( u^{k-2} e^{2 \Ld} \uu^2 \big) (t',y) dy
\bigg)
\uu (t',x)
\\
&\quad
+
\frac 1{4i}
u(t',0)^{2k}
\uu (t',x)
\bigg\}
dt',
\end{align*}
and $\Ld$ is given in \eqref{Ld1}.

Note that
\[
\| \Ld \|_{X_T}
\les \sup_{0 \le t \le T} \big( \| \Ld(t) \|_{L^2} + \| \dx \Ld (t) \|_{L^2} \big)
\les
\| u \|_{X_T}^k
\]
for $u \in X_T$.
Moreover, we have
\begin{equation}
\begin{aligned}
\| e^{\Ld}-1 \|_{X_T}
&\les
\int_0^1
\| \Ld e^{ \theta \Ld} \|_{X_T} d\theta
\les
e^{C \| u \|_{X_T}^k}
\| u \|_{X_T}^k
\end{aligned}
\label{expb}
\end{equation}
for $u \in X_T$.
It follows from
\eqref{expb} and $0 < T \le 1$ that
\begin{align*}
\| \G_1(u,\uu) \|_{X_T}
- \| \phi \|_{H^1}
&\le
\| u^k e^{\Ld} \uu \|_{X_T}
\les
\| u \|_{X_T}^k
\| \uu \|_{X_T}
\big(
1+ 
e^{C \| u \|_{X_T}^k}
\| u \|_{X_T}^k \big),
\\
\| \G_2(u,\uu) \|_{X_T}
- \| \psi \|_{H^1}
&\les
\| u^{k-1} e^{\Ld} \uu^2 \|_{X_T}
+ \| u^{k-1} e^{\Ld} \uu \|_{X_T}
\\
&
\quad
+ (k-1) \|u^{k-2} e^{2\Ld} \uu^2 \|_{X_T}
+ \| u^{2k} \uu \|_{X_T}
\\
&\les
\big(
\| u \|_{X_T}^{k-1} \| \uu \|_{X_T}^2
+
\| u \|_{X_T}^{k-1} \| \uu \|_{X_T}
\\
&\quad
\qquad
+
(k-1)
\| u \|_{X_T}^{k-2}
\| \uu \|_{X_T}^2
\big)
\big(
1+ 
e^{2 C \| u \|_{X_T}^k}
\| u \|_{X_T}^k \big)
\\
&\quad
+
\| u \|_{X_T}^{2k}
\| \uu \|_{X_T}
\end{align*}
for $u,\uu \in X_T$.
Note that the terms with the coefficient $(k-1)$ disappear when $k=1$.
Similarly,
we obtain difference estimates.
Hence,
$\G$ is a contraction mapping on
\[
\{ (u,\uu) \in X_T \times X_T \mid
\| u \|_{X_T} \le 2 \| \phi \|_{H^1}, \
\| \uu \|_{X_T} \le 2 \| \psi \|_{H^1}
\}
\]
provided that $\phi, \psi \in H^1_+(\T)$ satisfy $\| \phi \|_{H^1} + \| \psi \|_{H^1} \ll 1$.
The uniqueness follows from a standard argument.
\end{proof}

\begin{proof}[Proof of Proposition \ref{prop:welH+} for $\al=2$]
Let $u \in C([0,T]; H^2_+(\T))$ be a solution to \eqref{NLSk} with $\al=2$.
Define $\uu$ by \eqref{uug}.
Then, $(u, \uu)$ solves \eqref{gasys}.
Set
\[
\psi(x)
= e^{- \frac 1{2i} \int_0^x \phi(y)^k dy}
\dx \phi(x).
\]
It follows from \eqref{expb} that
\[
\| \psi \|_{H^1}
\les
\| \phi \|_{H^2} \big( 1+ e^{C \| \phi \|_{H^1}^k} \| \phi \|_{H^1}^k \big).
\]
Accordingly,
we have $\| \psi \|_{H^1} \ll 1$
provided that $\| \phi \|_{H^2} \ll 1$.
Hence, the uniqueness of $u$ follows from Proposition \ref{prop:gagesys2}.
Since we apply the contraction argument to prove Proposition \ref{prop:gagesys2},
the condition
\eqref{supp0} follows from the same reason as in the proof of Proposition \ref{prop:welH+} for $\al \ge 3$.
\end{proof}

\section{Norm inflation}
\label{sec:NI}

In this section,
we prove Theorem \ref{th:1.1}.
Let ($\al =2$ or $\al \ge 3$) and $k \in \N$.
Suppose that 
$u \in C([0,T]; H^{s_0}_{\ge 0}(\T))$ is a solution to \eqref{NLSk}
for some $T>0$,
where $s_0$ is defined by \eqref{s0}.
First,
by taking a transformation,
we consider a solution with mean-zero.

It follows from Remark \ref{rem:sol1} that
$\dt \ft u(t,0) =0$.
Namely, we have
\begin{equation}
\ft u(t,0) = \ft \phi (0)
\label{cons1}
\end{equation}
for $0 \le t \le T$.
For simplicity, we set
\[
\Aa
:=
\ft{\phi}(0).
\]

Assume that
\begin{equation}
\Im \Aa^k <0.
\label{conM0}
\end{equation}
Set
\begin{equation}
\tu (t,x)
:=
\sum_{n=0}^\infty  \ft u(t,n) e^{-it \Aa^k n} e^{inx}
- \Aa 
\label{truv}
\end{equation}
for $0 \le t \le T$ and $x \in \T$.
Note that
\eqref{conM0} yields that
\[
\Re (-i \Aa^k ) = \Im \Aa^k  <0.
\]
Hence,
$\tu$ is well-defined
and
\[
\tu \in C([0,T]; H^{s_0}_+(\T)) \cap C^\infty((0,T) \times \T).
\]

It follows from \eqref{cons1} that
\[
\ft \tu (t,0)
= \ft u(t,0) - \Aa 
=0.
\]
A direct calculation shows that
\begin{align*}
&\Ft \big[ \dt \tu - i (-\dx^2)^{\frac \al2} \tu \big] (t,n)
\\
&=
\Ft
\big[ \dt u - i (-\dx^2)^{\frac \al2} u  - \Aa^k \dx u \big]
(t,n)
e^{-i t \Aa^k n}
\\
&=
\Ft \big[ (u^k - \Aa^k) \dx u \big] (t,n)
e^{-i t \Aa^k n}
\\
&=
\sum_{j=1}^k
\Ft \bigg[
\begin{pmatrix} k \\ j \end{pmatrix}
\Aa^{k-j} (u - \Aa)^{j} \dx u
\bigg] (t,n)
e^{-i t \Aa^k n}
\\
&=
\sum_{j=1}^k
\begin{pmatrix} k \\ j \end{pmatrix}
\Aa^{k-j}
\Ft[ \tu^{j} \dx \tu]
(t,n)
\end{align*}
for $0 \le t \le T$ and $n \in \N$.
Hence,
$\tu$ satisfies
\begin{equation}
\left\{
\begin{aligned}
&\dt \tu - i (-\dx^2)^{\frac \al2} \tu
=
\sum_{j=1}^k
\begin{pmatrix} k \\ j \end{pmatrix}
\Aa^{k-j} \tu^{j} \dx \tu, \\
&\tu|_{t=0} = \wt \phi,
\end{aligned}
\right.
\label{NLSkk}
\end{equation}
where
$\wt \phi := \phi - \Aa$.
Note that $\Ft [\wt \phi] (0)=0$.

\begin{proof}[Proof of Theorem \ref{th:1.1}]
For $s \ge s_0$ and $N\in \N$ with $N \ge 3$,
we take the initial data $\phi$ as
\[
\phi (x)
=
\frac{e^{i{\frac{3\pi}{2k}}} + N^{-s}e^{iNx}}{\log N}.
\]
Note that
\begin{equation}
\Aa^k = \frac{-i}{(\log N)^k}.
\label{conM1}
\end{equation}
Hence, \eqref{conM0} is satisfied.
Moreover,
we have
\begin{equation}
\| \phi \|_{H^{s_0}}
\le
\| \phi \|_{H^s}
\le \frac 2 {\log N}.
\label{inis}
\end{equation}

Let $\s \in \R$.
Set
\begin{equation}
T := (|\s-s|+1) \frac{(\log N)^{k+1}}{N}.
\label{tN}
\end{equation}
Suppose that $u \in C([0,T]; H^s(\T))$ is a solution to \eqref{NLSk}.
Then, $\tu$ defined in \eqref{truv} satisfies \eqref{NLSkk}.
Namely, we have
\begin{equation}
\begin{aligned}
\ft \tu (t,n)
= e^{it n^\al} \ft \phi(n)
+
\frac{in}{k+1}
\int_0^t
&
e^{i (t-t') n^\al}
\sum_{j=1}^k
\begin{pmatrix} k \\ j \end{pmatrix}
\Aa^{k-j}
\\
&
\times
\sum_{\substack{n_1, \dots, n_{j+1} \in \N \\ n_1+\dots+n_{j+1}=n}}
\prod_{\l=1}^{j+1}
\ft \tu (t',n_\l)
dt'
\end{aligned}
\label{ftwn}
\end{equation}
for $0 \le t \le T$ and $n \in \N$.

By \eqref{inis} and taking $N \gg 1$,
the assumption in Corollary \ref{cor:UU} holds true.
Then, for $0 \le t \le T$ and $N \gg 1$,
we have
\[
\ft \tu (t,n) =0
\]
unless $n= mN$ for some $m \in \N$.
Hence,
the second term on the right-hand side of \eqref{ftwn} vanishes when $n=N$.
In particular,
we obtain
\begin{equation}
|\ft \tu(t,N)|
=
|\ft \phi(N)|
=
\frac{N^{-s}}{\log N}
\label{wtN}
\end{equation}
for $0 \le t \le T$ and $N \gg 1$.
It follows from \eqref{truv}, \eqref{conM1}, and \eqref{wtN} that
\begin{equation}
\begin{aligned}
\| u(T) \|_{H^\s}
&\ge
N^\s |\ft u(T,N)|
=
N^\s \big| \ft \tu (T,N) e^{i T \Aa^k N} \big|
\\
&=
\frac{N^{\s-s}}{\log N} \cdot N^{|\s-s|+1}
\ge \frac N{\log N}
\end{aligned}
\label{nft1}
\end{equation}
for $N \gg 1$.

For any $\eps>0$,
there exists $N \in \N$ with $N \gg 1$ and
\[
\frac 2{\log N} <\eps, \quad
(|\s-s|+1) \frac{(\log N)^{k+1}}{N} <
\min( \eps, 1), \quad
\frac N{\log N} > \eps^{-1}.
\]
From \eqref{inis}, \eqref{tN}, and \eqref{nft1},
we obtain Theorem \ref{th:1.1}.
\end{proof}

\mbox{}

\noindent
{\bf 
Acknowledgements.}
M.O.~was supported by JSPS KAKENHI Grant number JP23K03182.

\end{document}